\documentclass{imanum}
\jno{drnxxx}
\usepackage{mathptmx, amsmath, amsfonts}
\usepackage{natbib}
\usepackage{graphicx}
\usepackage{amsthm,mathrsfs,color}
\newcommand{\iprod}[1]{\langle#1\rangle}
\newcommand{\se}{\setcounter{equation}{0}}
\newcommand{\Grid}{\mathcal{G}}
\newcommand{\pw}{\varPi_{{k}}}
\newcommand{\p}{{\mathcal P}}
\newcommand{\cT}{\mathcal{T}}
\newcommand{\W}{{\mathcal W}}
\newcommand{\vhs}{V_h^*}
\newcounter{bean}
\newenvironment{romenum}{\begin{list}{{(\roman{bean})}}
{\usecounter{bean}}}{\end{list}}
\newcommand{\cL}{\mathcal{L}}
\newcommand{\Ba}{\mathcal{B}^\alpha}
\newcommand{\I}{\mathcal{I}}
\newcommand{\Jmas}{\mathcal{I}^{1-\alpha}_{ad}\,}
\begin{document}
\title{Finite volume element method for two-dimensional fractional  subdiffusion problems\thanks{The support of  the Science Technology Unit at KFUPM through  King Abdulaziz City for Science and
Technology (KACST) under National Science, Technology and Innovation Plan (NSTIP) project No. 13-MAT1847-04
 is gratefully acknowledged.}}

\author{\sc Samir Karaa \thanks{
Department of Mathematics and Statistics, Sultan Qaboos University,
P. O. Box 36, Al-Khod 123,
Muscat, Oman.
 Email: skaraa@squ.edu.om},
\sc Kassem Mustapha \thanks{
 Department of Mathematics and Statistics, King Fahd University
of Petroleum and Minerals, Dhahran, 31261, Saudi Arabia. Email: kassem@kfupm.edu.sa}
{\sc and}
\sc Amiya K. Pani{\thanks{
Department of Mathematics, Industrial Mathematics Group, Indian
Institute of Technology Bombay,
Powai, Mumbai-400076
Email: akp@math.iitb.ac.in}}}
\maketitle
\pagestyle{headings} \markboth{S. Karaa,
K. Mustapha  and A.  Pani}{\rm FINITE VOLUME METHODS FOR FRACTIONAL DIFFUSION PROBLEMS}

\begin{abstract}
{In this paper, a semi-discrete  spatial finite volume
(FV) method is proposed and analyzed for approximating solutions of anomalous subdiffusion equations
involving a temporal fractional derivative of order $\alpha \in (0,1)$ in a two-dimensional convex polygonal domain. Optimal error estimates in $L^\infty(L^2)$-
norm is shown to hold. Superconvergence result is proved and as a consequence, it is established that
quasi-optimal order of convergence in $L^{\infty}(L^{\infty})$ holds. We also consider a fully discrete
scheme that employs  FV method   in
space, and  a piecewise linear discontinuous Galerkin method  to discretize in temporal direction.  It is, further, shown that convergence rate is of order
$O(h^2+k^{1+\alpha}),$ where $h$ denotes the space discretizing parameter and $k$ represents the temporal
discretizing parameter. Numerical
experiments indicate optimal convergence rates in both time and space, and also illustrate that the imposed regularity assumptions are pessimistic. }
{Fractional diffusion equation,  finite volume element, discontinuous Galerkin method, variable meshes, convergence analysis}\end{abstract}

\section{Introduction}
\noindent
Let $\Omega$ be a bounded, convex polygonal domain in $\mathbb{R}^2$  with
boundary $\partial \Omega$, and let $f$ and $u_0$ be given functions defined on their
respective domains. Consider the  subdiffusion equation:
\begin{subequations}\label{eq: fractional diffusion}
\begin{alignat}{2}\label{a1}
&u'(x,t)+\Ba \cL u(x,t)=f(x,t) &&\quad\mbox{ in }\Omega\times (0,T],
\\  \label{a2}
&u(x,t)= 0 &&\quad\mbox{ on }\partial\Omega\times (0,T],
\\   \label{a3}
&u(x,0)=u_0(x) &&\quad\mbox{ in }\Omega,
\end{alignat}
\end{subequations}
where $\cL u=-\Delta u$, $u'$ is the partial derivative of
  $u$ with respect to time, $\Ba:={^R}{\rm D}^{1-\alpha}$ is the Riemann--Liouville  fractional
derivative in time defined by:   for $0<\alpha<1$,
\begin{equation} \label{Ba}
\Ba \varphi(t):=\frac{\partial }{\partial t}\I^{\alpha}\varphi(t):=\frac{\partial }{\partial t}\int_0^t\omega_{\alpha}(t-s)\varphi(s)\,ds\quad\text{with} \quad
\omega_{\alpha}(t):=\frac{t^{\alpha-1}}{\Gamma(\alpha)}
\end{equation}
with $\I^{\alpha}$ being  the temporal Riemann--Liouville fractional integral operator of order $\alpha$.

Fractional  diffusion models  received considerable attention over the last two decades from both practical and theoretical point of view. Researchers have found numerous porous media systems in which some key underlying random motion conform to a model where the diffusion is not classical, it is instead anomalously slow (fractional subdiffusion) or fast (super-diffusion). For example, the fractional diffusion problem \eqref{eq: fractional diffusion} captures the dynamics of some subdiffusion processes, where the growth of the mean square displacement  is slower compared to a Gaussian process, see Podlubny (1999) for more detail.  The modeling of this problem is actually  based on continuous time random walks  and master equations with power law waiting time densities (Henry \& Wearne (2000)), where $u$ represents the probability density function for finding a particle at location $x$ and at time $t$ (with waiting time and the jumps that are statistically independent). Fractional diffusion models  have been successfully used to describe diffusion in several phenomena including  media with fractal geometry (Nigmatulin (1986)), highly heterogeneous aquifer (Adams \& Gelhar (1992)),  and underground environmental problem (Hatano \& Hatano (1998)).

Many authors have proposed various techniques for approximating the solution $u$  of
 \eqref{eq: fractional diffusion}, however obtaining sharp error bounds under reasonable
 regularity assumptions on $u$ has proved challenging. Several types of finite difference schemes
 (implicit and explicit)  were investigated; see   Chen et al. (2012), Cuesta et al. (2006), Cui (2009), Langlands \& Henry (2005), Mustapha (2011),
 Quintana-Murillo \& Yuste (2013), Zhang et al. (2014) and related reference, therein.  The error analyses in most
 of these papers typically assume
that the solution $u$  is sufficiently smooth, including at $t = 0$. This enforces imposing compatibility
conditions on the given data.
In earlier works on time-stepping discontinuous Galerkin (DG)  method (including $hp$-versions)
combined with spatial standard Galerkin method by the second author and
McLean (McLean \& Mustapha (2009, 2015), Mustapha (2015), Mustapha \& McLean (2013)),
unbounded time derivatives of $u$ as $t \to 0$ was allowed (which is typically the case)
in the error analysis, also the case of non-smooth initial data was included. Variable time steps were  employed to compensate the singular behavior
of $u$, and consequently maintain optimal order rates of convergence.

Our main aim  is to propose  and analyze a method
using exact integration in time and  finite volume (FV)
method     for the space discretization for the two-dimensional fractional model \eqref{eq: fractional diffusion}. Then,  we  combine the FV scheme  in space with a piecewise-linear time-stepping DG scheme  which will then define a fully-discrete scheme. Compared to finite differences and finite elements, FV method is easier to implement on structured as well as unstructured
meshes and offers flexibility in tackling domains with complex boundaries. Further, it ensures local conservation
property of the fluxes which makes this method more attractive in applications. The approach followed here
is to formulate the problem in  the Petrov-Galekin frame  using two different  meshes to define
the trial space and test space, see Bank \& Rose (1987), Cai (1991)  and S\"uli (1991)
for some earlier results in this direction. This frame work helps us to derive error estimates
which are similar in spirit to tools developed for the error analysis of finite element method.
The choice of the FV method  for the problem under consideration is as used in Chatzipantelidis et al. (2004), Ewing et al. (2000), and
Chou \& Li (2000).

The major  contribution of the present article can be summarized as follows.
We first prove  that, under certain regularity assumptions on $u$ of problem
\eqref{eq: fractional diffusion},  the error of the FV approximation
to the solution $u$   in the
$L^\infty(L^2)$-norm (that is, $L^\infty\bigr(0,T;L^2(\Omega)\bigr)$-norm)  converge with order $h^2$, where $h$  is the maximum diameter
of the elements of the spatial mesh; see Theorem \ref{H2}. The imposed regularity conditions on $u$
can be satisfied by imposing some compatibility conditions on the given data taking into consideration
that the derivative of $u$ is not bounded near $t=0$, see the discussion after Theorem \ref{H2}.
In addition, under more restrictive regularity assumptions, we show errors of order $h^2$ in
the stronger $L^\infty(L^\infty)$-norm, see Theorem \ref{sup-conv-2}.
Since in the limiting case $\alpha \to 0$ the problem \eqref{eq: fractional diffusion} reduces to
the classical heat problem,  these convergence results extend those obtained in Chatzipantelidis et al. (2004, 2009)
and Chou \& Li (2000) for the heat equation. This extension is indeed not straightforward,  we make the
full use of several important properties of the fractional derivative operator and also use clever
steps (see for instance the proof of Lemma \ref{sup-conv-1}) to achieve our goal.
In the second part of the paper, we derive the error from the fully-discrete scheme (DG in time and  FV in space) for \eqref{eq: fractional diffusion}. In the $L^\infty(L^2)$-norm,
we show convergence of order $h^2+k^{1+\alpha}$ (that is, suboptimal in time) where $k$ is the maximum
time step-size. Proving this rate of convergence in the stronger $L^\infty(L^\infty)$-norm
is beyond the scope of the paper due to several technical difficulties. It is worthy to mention that the numerical results demonstrate optimal convergence rate in both time and space in the $L^\infty(L^\infty)$-norm,
and also shows that the regularity conditions on $u$ are pessimistic. In this regard, the approach
used in the time-stepping DG error analysis in Mustapha (2015) might be beneficial to prove a better
convergence rate  in time,  $k^{\frac{3+\alpha}{2}}$ instead of $k^{1+\alpha}$.

An outline of the paper is as follows. In the next section, we introduce some notations and state some important properties of the time fractional operator $\Ba$.  In Section \ref{sec: FVM}, we introduce our semi-discrete FV scheme in space for problem
  \eqref{eq: fractional diffusion} and define some
  interpolation operators that play an important role in our error analysis. Section \ref{sec: FV analysis}
  is devoted to prove the main convergence result from the FV discretization, Theorems  \ref{H2} and
  \ref{sup-conv-2}. Particularly relevant
to this {\it a priori} error analysis is the appropriate use of several important properties of
the operator $\Ba$. In Section \ref{sec: fully discrete numerical method}, we define our fully-discrete DG FV  scheme and show the corresponding convergence results   in the following section, see Theorem \ref{thm: fully discrete}. Finally, in Section 6, we present some numerical results to demonstrate our theoretical achievements and   illustrate optimal rates of convergence in both time and space  (not only in space  as the theory suggested) in the $L^\infty(L^\infty)$-norm  under weaker regularity assumptions than the theory required.

\section{Notation and Preliminaries.}
\se

Denote by $(\cdot ,\cdot)$ and $\|\cdot\|$  the $L^2$-inner product and its
induced norm on $L^2(\Omega)$, respectively. The $L^\infty(\Omega)$-norm is denoted by $\|\cdot\|_\infty$. Let $H^m(\Omega)=W^{m,2}(\Omega)$ denote the standard
Sobolev space equipped with the usual norm $\|\cdot\|_m$.  With
$H_0^1(\Omega)=\{v\in H^1(\Omega): v=0 \;\mbox{on}\;  \partial \Omega \},$
let $A(\cdot, \cdot):H_0^1 (\Omega)\times H_0^1(\Omega)\rightarrow\mathbb{R}$ be the bilinear form
associated with the operator $\cL$ which is symmetric and positive definite  on $H_0^1 (\Omega)$.
Then, the weak formulation for \eqref{eq: fractional diffusion} is to seek
$u:(0,T]\longrightarrow H^1_0(\Omega)$ such that
\begin{equation} \label{weak}
(u',v)+ A(\Ba u,v)=  (f,v)\quad
\forall ~v\in H_0^1(\Omega)\quad{with}~~u(0)=u_0.
\end{equation}
 Note that for $\varphi\in W^{1,1}(0,T)$, $\Ba$
satisfies the following property (Theorem A.1, McLean (2012)):
\begin{equation*}
\int_0^T \Ba \varphi(t)\,\varphi(t)\,dt
\geq  c_\alpha T^{\alpha-1}\int_0^T|\varphi(t)|^2\,dt\;\;\; \mbox{ for } \; 0 < \alpha < 1,
\end{equation*}
where $c_\alpha=\pi^{1-\alpha}(1-\alpha)^{1-\alpha}(2-\alpha)^{\alpha-2}
\sin(\alpha\pi/2)$ is a positive constant.

In contrast,  the Riemann--Liouville operator $\I^{\alpha}$ has also some positivity property but with a weaker  lower bound compared to the  one of the operator $\Ba.$ More precisely,
by Lemma 3.1(ii) in Mustapha \& Sch\"otzau (2014), it follows that for piecewise continuous functions $\varphi:[0,T] \to \mathbb{R},$
\begin{equation}\label{J-P}
\int_0^T\I^\alpha\varphi(t)\,\varphi(t)\,dt\geq
\cos(\alpha \pi/2)\int_0^T|\I^{\alpha/2}\varphi(t)|^2\,dt \geq 0~~{\rm  for}~~0<\alpha<1\,.
\end{equation}
Since the  bilinear form $A(\cdot, \cdot)$ is symmetric
positive definite,  the following holds: for
$W^{1,1}(0,T;H_0^1(\Omega))$,
\begin{equation}\label{positive definite property of Ba}
\int_0^TA(\Ba\varphi(t),\varphi(t))\,dt \geq c_\alpha T^{\alpha-1}
\int_0^T\| \nabla \varphi(t)\|^2\,dt.
\end{equation}
In the sequel, we shall use the
adjoint operator $ \I_{ad}^{\alpha}$ of $\I^{\alpha}$ (Lemma 3.1,  Mustapha \& McLean (2013)):
\begin{equation}\label{Bas}
\I_{ad}^{\alpha}\, \varphi(t)=
\int_t^T\omega_{\alpha}(s-t)
\varphi(s)\,ds \quad\mbox{for }
\varphi\in C^0[0,T]~~{\rm with}~~ 0<\alpha<1.
\end{equation}
For later use, we recall the following property (Section 3, Cockburn \& Mustapha (2015)):
\begin{equation}\label{identity}
\I^{1-\alpha}\bigl(\Ba\varphi\bigr)(t)=\varphi(t) \quad\mbox{for }
\varphi\in C^1(0,T).
\end{equation}

\section { Finite volume element method}\label{sec: FVM}
\se
This section deals with primary and dual meshes on the domain $\Omega$, construction of finite dimensional spaces, FV element formulation and some preliminary results.

Let $\cT_h$ be a family of regular (quasi-uniform) triangulations of
the closed, convex  polygonal domain $\overline{\Omega}$ into triangles  $K,$
and let $h=\max_{K\in \cT_h}h_{K},$ where $h_{K}$ denotes the diameter  of $K.$  Let $N_h$ be set of nodes or vertices, that is, $N_h :=\left\{P_i:P_i~~\mbox{ is a vertex of the element }~K \in
\cT_h~\mbox{and}~P_i\in \overline{\Omega}\right\}$ and let
$N_h^0$ be the set of interior nodes in $\cT_h.$
Further, let $\cT_h^*$ be the dual mesh associated with the primary mesh $\cT_h,$ which is defined as follows. With $P_0$ as an interior node of the triangulation $\cT_h,$ let  $P_i\;(i=1,2\cdots
m)$ be its adjacent nodes (see, Figure ~\ref{fig:mesh} with $m=6$ ). Let $M_i,~i=1,2\cdots
m$ denote the midpoints of $\overline{P_0P_i}$ and let $Q_i,~i=1,2\cdots
m,$  be  the barycenters of the triangle $\triangle P_0P_iP_{i+1}$ with
$P_{m+1}=P_1$. The {\it control volume}
  $K_{P_0}^*$ is constructed  by joining successively $ M_1,~ Q_1,\cdots
  ,~ M_m,~ Q_m,~ M_1$. With $Q_i ~(i=1,2\cdots
m)$ as the nodes of $control~volume~$ $K^*_{p_i},$ let $N_h^*$ be the set of all dual nodes
$Q_i$. For a  boundary
node $P_1$,  the control volume $K_{P_1}^*$ is shown in Figure ~\ref{fig:mesh}. Note that the union
of the control volumes forms a partition $\cT_h^*$ of $\overline{\Omega}$.

\begin{figure}
 \begin{center}
 \includegraphics*[width=10.0cm,height=5.0cm]{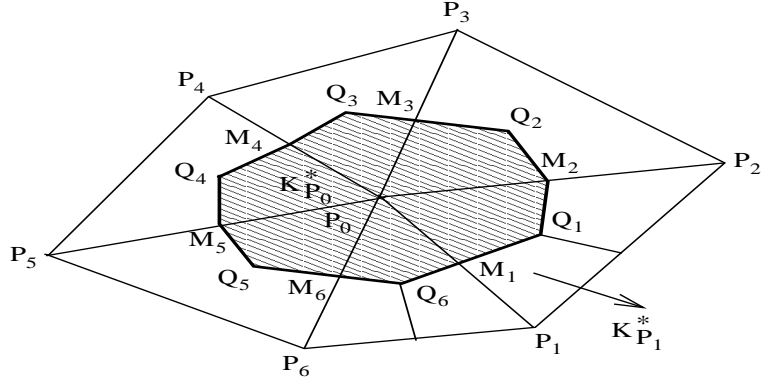}
    \caption{The control Volumes for the boundary nodes}
      \label{fig:mesh}
   \end{center}
\end{figure}

We consider a FV element  discretization of (\ref{eq: fractional diffusion})
in the standard $C^0$-conforming piecewise linear  finite element space $V_h$ on the primary mesh $\cT_h$, which is defined by
$$V_h=\{v_h\in C^0(\overline {\Omega})\;:\;v_h|_{K}\;\mbox{is linear for all}~ K\in \cT_h\; \mbox{and} \; v_h|_{\partial \Omega}=0\},$$
and  the dual volume element space $\vhs$ on the dual mesh $\cT^*_h$ given by
$$\vhs=\{v_h\in L^2(\Omega)\;:\;v_h|_{K_{P_0}^*}\;\mbox{is constant for all}\; K_{P_0}^*\in \cT_h^*\; \mbox{and}\; v_h|_{\partial \Omega}=0\}.$$

The semi-discrete FV element  formulation for \eqref{eq: fractional diffusion} is
to seek  $u_h:(0,T]\longrightarrow V_h$ such that
\begin{equation} \label{semi}
(u_h',v_h)+ A_h(\Ba u_h,v_h)=  (f,v_h)\quad
\forall v_h\in \vhs ~~~~~~~~~~~~~~~~~~~~~
\end{equation}
with given  $u_h(0)\in V_h$ to be defined later.
Here, $A_h(\cdot, \cdot)$ is  defined by
\begin{equation} \label{Ah}
A_h(w_h,v_h)= -\sum_{P_i\in N_h^0} v_h(P_i) \int_{\partial K_{P_i}^*}
\nabla w_h\cdot{\bf n}\,ds\quad \forall~ w_h\in V_h~{\rm and}~v_h\in  \vhs\,,
\end{equation}
 with ${\bf n}$ denoting the outward unit normal to the boundary of the control volume
$K_{P_i}^*$.  For $v\in H^2(\Omega)$, a use of Green's formula
yields
\begin{equation}\label{eq: h2 identity}
(\mathcal{L}v,v_h)=A_h(v,v_h),\quad \forall~v_h\in  \vhs\,.
\end{equation}
Moreover as $\mathcal{L}=-\Delta,$ the following identity holds:
\begin{equation}\label{sup-1}
A(w_h,\chi)=A_h(w_h,\Pi_h^\ast\chi)\qquad \forall w_h,\;\chi\in V_h.
\end{equation}
Hence, taking  the $L^2$-inner product of (\ref{eq: fractional diffusion}) with
$v_h\in \vhs$ yields
\begin{equation}\label{consistency}
(u',v_h)+ A_h(\Ba u,v_h) =   (f,v_h)\quad\forall ~v_h \in \vhs.
\end{equation}

For the error analysis, we first introduce two interpolation operators.
Let $\Pi_h:C^0(\bar{\Omega})\longrightarrow V_h$ be the piecewise linear interpolation operator and
$\Pi_h^*:C^0(\bar{\Omega})\longrightarrow \vhs$ be the piecewise constant interpolation operator. These interpolation operators are defined,  respectively,  by
\begin{equation}
\Pi_h u=\sum_{P_i\in N_h^0}u(P_i)\phi_i(x)\;~\mbox{and}~\; \Pi_h^*u=\sum_{P_i\in N_h^0}u(P_i)\chi_i(x).
\label{naa}
\end{equation}
It is known that   $\Pi_h$ has the following approximation property (Ciarlet (1978)):
\begin{equation}
\|\psi-\Pi_h\psi\|_0\leq Ch^2\|\psi\|_2\quad {\rm for}\quad \psi \in H^2(\Omega)\,.
\label{1.6}
\end{equation}
We state next  the properties of
the interpolation operator $\Pi_h^*.$ For a proof, see  (pp. 192, Li et al. (2000)). For convenience,
we introduce the following notations:  $  \langle \phi, v\rangle := ( \phi,\Pi_h^*v )$ for $\phi \in L^2$ and $v\in C^0(\bar{\Omega})$.

\begin{lemma}\label{lem: norm equivlence}
\label{E}
 The following statements  hold true.
 \begin{romenum}
\item
 $\langle \cdot, \cdot\rangle$ defines an inner-product on
  $V_h \times V_h$ with its induced norm denoted by $ |\|\cdot|\|.$
\item The norms  $|\|\cdot|\|$  and
$\|\cdot\| $  are equivalent on $ V_h$\,.
\item The operator $\Pi^*_h$ is stable in the following sense: $ \|\Pi_h^* \chi\| \leq C \|\chi\|$ for any $\chi \in V_h.$
 \end{romenum}
\end{lemma}
\section{ A Priori Error Estimates}\label{sec: FV analysis}
\se
This section deals with  {\it a priori} optimal order error estimates for the semi-discrete
FV scheme (\ref{semi}). To do so,
we split the error    as:
$$u_h-u=(u_h-\bar R_h u)+(\bar R_h u-u)=:\theta+\xi,$$
where  $\bar R_h :H_0^1(\Omega)\cap H^2(\Omega) \rightarrow V_h$ is the finite volume elliptic projection
operator defined by
\begin{equation}\label{bar Ritz}
A_h(\bar R_h v,\chi)= A_h(v,\chi)\qquad \forall\, \chi\in V_h^*.
\end{equation}
For each $t\in (0,T]$, the projection error $\xi(t)$ satisfies the following estimates (Chou \& Li (2000)):
\begin{eqnarray}\label{rho-estimate}
\|\xi(t)\|_\ell + \|\xi'(t)\|_\ell \leq C h^{2-\ell} \Big(\|u(t)\|_{W^3_p} + \|u'(t)\|_{W^3_p}\Big) \quad {\rm for}
~\ell=0,1 \quad {\rm with}  ~~p>1.
\end{eqnarray}
Moreover, the following maximum estimate is also valid for $t\in (0,T]$
\begin{equation}\label{rho-estimate-infty}
\|\xi(t)\|_{\infty}\leq C h^2 |\log\left( h\right)|\Big(\|u(t)\|_3+ \|u(t)\|_{W^{2,\infty}(\Omega)} \Big).
\end{equation}

\noindent
From  (\ref{semi}) and \eqref{consistency}, it follows that
 \begin{equation} \label{theta}
\iprod{\theta',\chi}+A_h(\Ba \theta,\Pi_h^*\chi) =
-(\xi',\Pi_h^*\chi) \quad \forall~\chi \in V_h.
\end{equation}
On substitution of (\ref{sup-1}) in (\ref{theta}),  we obtain
\begin{equation} \label{sup-2}
\iprod{\theta',\chi}+A(\Ba\theta,\chi)=
-(\xi',\Pi_h^*\chi) \quad \forall~\chi \in V_h\,.
\end{equation}

Below, we prove one of the main theorems of this section.
We may  choose
$u_h(0)= \Pi_h u_0,$ or even $L^2$-projection onto $V_h,$ then, using approximation property and the equivalence
of norms (ii) of Lemma \ref{lem: norm equivlence}, it follows that
$\||\theta(0)\|| \leq C\|\theta(0)\| \leq C h^2 \; \|u_0\|_2.$
In case, we choose  $u_h(0)=\bar R_h u_0,$ then $\theta(0)=0.$
\begin{theorem}
\label{H2}
Let $u$ and $u_h$  be the solutions of $(\ref{eq: fractional diffusion})$ and $(\ref{semi}),$
respectively. Further, let $u_h(0)$ be chosen so that
$\|u_0- u_h(0)\| = O(h^2).$
Then
for any $T>0$, there is a positive constant $C,$ which may depend on  $T$ and $\alpha,$ but independent of $h$
such that
$$
 \|(u_h-u)(T)\| \leq C\, h^2\left(\|u_0\|_3+\|u'\|_{L^1(H^3)}\right).
$$
\end{theorem}
\begin{proof}  Since $u_h-u=\xi+\theta,$ where the estimate of $\xi$ is given in \eqref{rho-estimate},
it is sufficient to estimate $\theta.$ To this end,   choose  $\chi=\theta$ in
(\ref{sup-2}) and obtain
\begin{equation}\label{k}
\iprod{\theta',\theta}+A(\Ba\theta,\theta)=
-(\xi',\Pi_h^*\theta).
\end{equation}
Integrating  from $0$ to $t$ and using $
\iprod{\theta',\theta}=\frac{1}{2}\frac{d}{dt} |\|\theta|\|^2$ yield
\begin{equation}\label{1.27}
\||\theta(t)\||^2+ 2\int_0^tA(\Ba\theta,\theta)\,ds \le |\|\theta(0)|\|^2+2\int_0^t |(\xi',\Pi_h^*\theta)(s)|\,ds\,.
\end{equation}
By the stability property  for the operator $\Pi_h^{*}$ in Lemma \ref{lem: norm equivlence} (iii),
and the  equivalence of norms  in Lemma \ref{lem: norm equivlence} (ii), we have
 $|(\xi',\Pi_h^*\theta)| \leq   \,\|\xi'\| \; \|\Pi_h^{*} \theta \|
 \leq C \|\xi'\| \; \||\theta \||.$
On substitution this  in \eqref{1.27}, and use the positivity property of $\Ba$
in \eqref{positive definite property of Ba} to obtain after simplification
\begin{equation}\label{theta-1}
\||\theta(t)\||^2 \leq \||\theta(0)\||^2 +
C\,\int_{0}^{t} \|\xi'(s)\|\;\||\theta(s)\||\;ds.
\end{equation}
Let $t^*\in [0,t]$ be such that
$
\||\theta(t^*)\|| := \max_{0\leq s\leq t} \,\||\theta(s)\||.$
Then, it is easy to check from (\ref{theta-1}) that
\begin{equation}\label{theta-2}
\||\theta(t)\|| \leq \||\theta(t^*)\|| \leq \||\theta(0)\|| + C\,\int_{0}^{t} \|\xi'(s)\|\;ds\,.
\end{equation}
Therefore, the desired error  estimate follows from the decomposition $u_h-u=\xi+\theta$, the inequality
$\|\theta(T)\|\le C\, |\|\theta(T)|\|$
by Lemma \ref{lem: norm equivlence} (ii),
the above bound, the finite volume elliptic
 projection error (\ref{rho-estimate}),  and the inequality $\|u(T)\|_3\le \|u_0\|_3+\int_0^T \|u'(t)\|_3\,dt.$
This completes the rest of the proof.
\end{proof}

Due to the singular behaviour of the solution $u$ of \eqref{eq: fractional diffusion} near $t=0$, some regularity and compatibility assumptions on the given data $u_0$ and $f$ are required to make sure that the term $\|u_0\|_3+\|u'\|_{L^1(H^3)}$ is bounded. Consequently, by Theorem \ref{H2},  the error for the spatial discretization by FV method  is of order $O(h^2)$. For instance, if $f\equiv 0$, we assume that $u_0 \in H^4(\Omega) \cap H^1_0(\Omega)$ and
$\cL u_0  \in H^1_0(\Omega)$, then by (Theorem 4.2, Mclean (2010)),
\begin{equation}\label{eq: regularity 1}
\|u'(t)\|_2+t^{\frac{\alpha}{2}}\|u'(t)\|_3 
\le C_1 \,t^{\alpha-1}\quad {\rm for}~~t>0\,.\end{equation}
Hence,  from Theorem \ref{H2}, $ \|(u_h-u)(t)\| \leq C\, h^2$ for $t>0.$  Here, we can argue that the assumptions on $u_0$ can be slightly relaxed. Instead, we assume that $u_0 \in H^3(\Omega) \cap H^1_0(\Omega)$ and $\cL u_0  \in H^1_0(\Omega)$, and again by (Theorem 4.2, Mclean (2010)), we arrive at
\begin{equation}\label{eq: regularity 2}
\|u'(t)\|_1+t^\frac{\alpha}{2}\|u'(t)\|_2+t^\alpha\|u'(t)\|_3 \le C\,t^{\alpha-1}\quad {\rm for}~~t>0\,.\end{equation}
Using the elliptic projection bound $\|\xi'(t)\|\leq C h \|u'(t)\|_2,$
and also the projection bound in \eqref{rho-estimate}, we observe for $\epsilon>0$ that
\begin{align*}
 \|\xi(t)\| &\leq \|\xi(0)\|+\int_0^t \|\xi'(s)\|\,ds\\
 &\le   C\, h^2\left(\|u_0\|_3+h^{-1}\int_0^\epsilon \|u'(s)\|_2\,ds+\int_\epsilon^t\|u'(s)\|_3\,ds\right)\quad{\rm for}~~t>0.
\end{align*}
Thus, by the regularity property \eqref{eq: regularity 2}, it follows that
\begin{equation}\label{eq: bound of xi}
\begin{aligned}
 \|\xi(t)\| &\leq C\, h^2\left(\|u_0\|_3+h^{-1}\int_0^\epsilon s^{\frac{\alpha}{2}-1}\,ds+\int_\epsilon^t s^{-1}\,ds\right)\\
& \leq C\, h^2\left(1+h^{-1}\epsilon^{\frac{\alpha}{2}}+ |{\rm log}(t/\epsilon)|\right)\\
&\le C\, h^2\left(1+ |{\rm log}\; h|\right)\quad{\rm with}~~ \epsilon={\min\{h^{\frac{2}{\alpha}},t\}}\,.
\end{aligned}
\end{equation}
Therefore,  a use of the obtained estimate of $\xi(t)$ in the proof of the Theorem \ref{H2} yields
\[
 \|(u_h-u)(t)\| \le  C\, h^2\left(1+ |{\rm log}\; h|\right).\]

Our next aim is to derive an estimate of order $O(h^2),$ but in the stronger $L^\infty(L^\infty)$-norm. To do so, we start by
  estimating $\theta'$ in the next lemma. This bound is needed
for showing  the super-convergence result of $\nabla\theta$ in
$L^\infty(L^2)$-norm. 

\begin{lemma}\label{lem-sup-1}
With $u_h(0)= \bar R_h u_0,$ there exists a positive constant $C$ independent of $h$ such that
\[
\int_0^t\|\theta'\|^2\,ds \leq C \int_0^t\|\xi'\|^2ds\quad{\rm  for}~~ 0<t\le T\,.\]
\end{lemma}
\begin{proof}
Choose $\chi=\theta'$ in (\ref{sup-2}) to obtain
\[
|\|\theta'|\|^2+A(\Ba\theta,\theta')=
-(\xi',\Pi_h^*\theta').
\]
Since $\theta(0)=0$, $\Ba \theta=\I^\alpha \theta'.$ Using this  and the positivity property  of $\I^\alpha$, \eqref{J-P},
we arrive at
\begin{equation*}
\int_0^tA(\Ba\theta,\theta')\,ds=\int_0^t (\I^\alpha(\nabla\theta'),\nabla\theta')\, ds
\geq 0\,.
\end{equation*}
Hence, by the Cauchy-Schwarz inequality,  the stability property of $\Pi_h^*$, and
also by Lemma \ref{lem: norm equivlence} $(ii)$,
\begin{equation*}
\int_0^t|\|\theta'|\|^2\,ds\le
\int_0^t|(\xi',\Pi_h^*\theta')|\,ds \leq C
\int_{0}^{t} \|\xi'\|\,\|\theta'\|\;ds \leq C \int_{0}^{t} \|\xi'\|\,\||\theta'|\|\;ds.
\end{equation*}
Therefore, a use of  Holder's inequality and again Lemma \ref{lem: norm equivlence} $(ii)$
complete the rest of the proof.  \end{proof}

Below, we derive an upper bound of $\nabla\theta$ in
$L^\infty(L^2)$-norm.
\begin{lemma}\label{sup-conv-1}
Assume that the solution $u$ of \eqref{eq: fractional diffusion} satisfies the regularity property \eqref{eq: regularity 1}.
 With $u_h(0)=\bar R_h u_0,$ there exists a positive constant $C,$ independent of $h,$ such that
\[
\|\nabla\theta(t)\|^2
\leq Ch^2
\int_0^t \| \theta'\|\,ds+Ch^4 |\log(t/{\min\{h^{\frac{2}{\alpha}},t\}})| .\]
\end{lemma}
 \begin{proof}
Choose $\chi=\Jmas\theta'$ in (\ref{sup-2}) and integrate
the resulting equation over the interval $(0,t)$. Use
(\ref{identity}) to arrive at
\begin{equation}\label{sup-9}
\int_0^t\iprod{\I^{1-\alpha}\theta',\theta'}\,ds+\int_0^tA(\theta,\theta')\,ds=
-\int_0^t(\xi',\Pi_h^*\Jmas\theta')\,ds.
\end{equation}
To bound  the term on the right hand side of (\ref{sup-9}), we decompose $\xi$ as
$$\xi=(\Pi_h u-u)+(\bar R_h u-\Pi_h u)=:\xi_1+\xi_2.$$ 
Since $\xi_2 \in V_h$ and since $\Pi_h^*$ commutes with $\Jmas$,
\begin{equation*}
\int_0^t(\xi',\Pi_h^*\Jmas\theta')\,ds
= \int_0^t(\I^{1-\alpha}\xi_1',\Pi_h^*\theta')\,ds+\int_0^t\iprod{\I^{1-\alpha}\xi_2',\theta'}\,ds
\end{equation*}
By the continuity property (see Lemma 3.1(iii) in Mustapha \& Sch\"otzau (2014)) of $\I^{1-\alpha}$,
\[\Big|\int_0^t\iprod{\I^{1-\alpha}\xi_2',\theta'}\,ds\Big|\le C_\alpha\int_0^t\iprod{\I^{1-\alpha}\xi_2',\xi_2'}\,ds
+\frac{1}{2}\int_0^t\iprod{\I^{1-\alpha}\theta',\theta'}\,ds\,.\]
Substitute the above equations  in (\ref{sup-9}) and use the equality $\int_0^tA(\theta,\theta')\,ds=\frac{1}{2}\|\nabla \theta(t)\|^2$ (because $\theta(0)=0$) to obtain after simplifying
\begin{equation}
\label{sup-11-1}
\int_0^t\iprod{\I^{1-\alpha}\theta',\theta'}\,ds+\|\nabla\theta(t)\|^2\le 2
\Big|\int_0^t(\I^{1-\alpha}\xi_1',\Pi_h^*\theta')\,ds\Big|+ C_\alpha\int_0^t\iprod{\I^{1-\alpha}\xi_2',\xi_2'}\,ds\,.
\end{equation}
By the stability of $\Pi_h^*,$ it follows that
\begin{equation*}
\begin{aligned}
 2\Big|\int_0^t(\I^{1-\alpha}\xi_1',&\Pi_h^*\theta')\,ds\Big|+ C_\alpha\int_0^t \iprod{\I^{1-\alpha}\xi_2',\xi_2'}\,ds\\
&\leq C\int_0^t\|\I^{1-\alpha}\xi_1'\|\,\|\theta'\|\,ds+ C_\alpha\int_0^t\|\I^{1-\alpha}\xi_2'\|\,\|\xi_2'\|\,ds\\
&\leq  C\int_0^t\int_0^s (s-\tau)^{-\alpha}\Big(\|\xi_1'(\tau)\|\,\|\theta'(s)\|+\|\xi_2'(\tau)\|\,
\|\xi_2'(s)\|\Big)
\,d\tau\,ds.
\end{aligned}
\end{equation*}
Now split $\xi_2$ as $\xi_2= \xi-\xi_1$ and use the  elliptic projection with \eqref{rho-estimate}
and \eqref{1.6} to arrive at
\begin{equation}\label{sup-11-2}
\begin{aligned}
& 2 \Big|\int_0^t(\I^{1-\alpha}\xi_1',\Pi_h^*\theta')\,ds\Big|+ C\int_0^t \iprod{\I^{1-\alpha}\xi_2',\xi_2'}\,ds\\
&\leq  Ch^2\int_0^t\int_0^{s} (s-\tau)^{-\alpha}\|u'(\tau)\|_2\,\|\theta'(s)\|\,d\tau\,ds \\
&+  Ch^3\Big( \int_0^\epsilon \|u'(s)\|_2
+h\int_\epsilon^t \|u'(s)\|_3 \Big) \int_0^{s} (s-\tau)^{-\alpha}\|u'(\tau)\|_3\,\,d\tau\,ds\\
&\leq  Ch^2 \int_0^t  \|\theta'(s)\|\,ds+
Ch^4\Big(h^{-1}\int_0^\epsilon  s^{\frac{\alpha}{2}-1} \,ds+\int_\epsilon^t s^{-1}\,ds\Big)\\
&\leq  Ch^2 \int_0^t  \|\theta'(s)\|\,ds+
Ch^4\;\Big(h^{-1}\;\epsilon^{\frac{\alpha}{2}}+\log \Big(\frac{t}{\epsilon} \Big)\Big)
\,.
\end{aligned}
\end{equation}
where in the second  last step, we have used the regularity assumption \eqref{eq: regularity 1} and
the inequalities:
\begin{align*}
&\int_0^s (s-\tau)^{-\alpha} \tau^{\frac{\alpha}{2}-1}\,d\tau\le
C\,s^{-\frac{\alpha}{2}}\quad{\rm and}\quad
\int_0^s (s-\tau)^{-\alpha} \tau^{\alpha-1}\,d\tau\le C\,.
\end{align*}
Substitute \eqref{sup-11-2} in \eqref{sup-11-1},
 choose $\epsilon={\min\{h^{\frac{2}{\alpha}},t\}}$, and  use the positivity property \eqref{J-P} of the operator $\I^{1-\alpha}$ to complete the rest of the proof.
 \end{proof}

As a consequence of the super-convergent result proved in  Lemma \ref{sup-conv-1}, we prove in
the next theorem,  the following  maximum norm convergence.
\begin{theorem}\label{sup-conv-2}
Let $u_h(0)= \bar R_h u_0.$ Assume that  $u \in H^1(0,T;H^3(\Omega))\cap L^\infty(0,T;W^{2,\infty}(\Omega))$.  Then,
\[
\|(u_h-u)(t)\|_\infty \leq Ch^2|{\rm log} \;h\;|\quad{\rm  for}~~t\in (0,T],\]
where the constant $C$ depends on  $T$  and $\alpha$, but
  independent of $h$\,.
  \end{theorem}
  \begin{proof}
From the decomposition $u_h-u=\xi+\theta$ and the estimates of
$\xi$ in \eqref{rho-estimate-infty}, we obtain for $t >0$
\[\|(u_h-u)(t)\|_\infty\le C h^2\;|{\rm log} \;h\;|\;\Big(\|u(t)\|_3+ \|u(t)\|_{W^{2,\infty}(\Omega)} \Big)+\|\theta(t)\|_\infty\,.\]
However, by Sobolev lemma when $\Omega\subset \mathbb{R}^2$ and Lemma \ref{sup-conv-1}, we  arrive for $t>0$ at
\begin{align*}
\|\theta(t)\|_\infty
&\leq C\;|{\rm log} \;h\;|\;\|\nabla\theta(t)\| \leq Ch\;|{\rm log} \;h\;|\;
\int_0^t \| \theta'\|\,ds+ Ch^2 \;| {\rm log}\;(t/{\min\{h^{\frac{2}{\alpha}},t\}})|\,.
\end{align*}
To complete the proof, we use the inequality  $\int_0^t \| \theta'\|\,ds\le  t^{1/2}\Big(\int_0^t \| \theta'\|^2\,ds\Big)^{1/2}$, Lemma \ref{lem-sup-1}, the  bound in \eqref{rho-estimate},
and the regularity assumption $u \in H^1(0,T;H^3(\Omega))$.
\end{proof}
\begin{remark} The  assumption $u\in H^1(0,T;H^3(\Omega))$ is stronger than the one imposed in \eqref{eq: regularity 1}. Noting that, under the regularity assumptions in  \eqref{eq: regularity 1} with $\alpha \in (1/2,1)$, one can show that the error in the $L^\infty(L^\infty)$ is of order $h^{\frac{3}{2}}$ (ignoring the logarithmic factor), that is, suboptimal. To see this, we follow the derivation in the above theorem, and use
\[\int_0^t \| \theta'\|^2\,ds\le C\,\int_0^t \|\xi'\|^2\,ds
\le C\,h^2\int_0^t \| u'\|_2^2\,ds\le C\,h^2\int_0^t s^{2\alpha-2}\,ds \le C\,h^2.\]
However,  our numerical
experiments illustrate that the imposed regularity assumptions are pessimistic. We observe optimal rates of convergence in the absence of the regularity  assumptions \eqref{eq: regularity 1}.
\end{remark}
\section{The fully-discrete numerical scheme}\label{sec: fully discrete numerical method}
\se
This section is devoted to our fully-discrete scheme   for the fractional diffusion model problem \eqref{eq: fractional diffusion}. We discretize in time using a piecewise linear DG  method (Mustapha \& McLean (2011)) and the FV method  for the spatial discretization.  To this end,   we introduce a (possibly non-uniform) partition of the time interval
$[0,T]$ given by the points: $0=t_0<t_1<t_2<\cdots<t_N=T,$ and define the half-open subinterval~$I_n=(t_{n-1},t_n]$ with length
$k_n=t_n-t_{n-1}$ for $1\le n\le N$ and set ~$k:=\max_{1\le n\le
N}k_n$.

Next, we introduce our {\em time-space} finite dimensional   spaces~$\W$ and $\W^*$ as
\begin{align*}
\W&:=\{  v\in L_2((0,T);V_h): v|_{I_n }\in P_1(V_h) , 1\leq n\leq N\},\\
\W^*&:=\{  v\in L_2((0,T);V_h^*): v|_{I_n }\in P_1(V_h^*) , 1\leq n\leq N\},
\end{align*}
where $P_1(S)$  denotes the space of linear polynomials in~$t$  with coefficients in a given space $S.$

The DG FV  approximation $U\in \W$ for (\ref{eq: fractional diffusion}) is now
defined as follows: Given $U(t)$ for $0\le t\le t_{n-1}$ with $U(0) :=U^0=U^0_+ \approx  u_0 $, the
 solution $U\in P_1(V_h)$ on $I_n$
is determined by requiring that for $1\le n\le N$,
\begin{equation}\label{eq: DG In}
(U^{n-1}_+,X^{n-1}_+)+\int_{I_n}\bigl\{(U',X)
    +A_h\bigl(\Ba U,X\bigr)\bigr\}\,dt
    =(U^{n-1},X^{n-1}_+)+\int_{I_n}(f,X)\,dt\quad \forall ~X\in P_1(V_h^*),
\end{equation}
 where
\[
U^n:=U(t_n)=U(t_n^-),\quad U^n_+:=U(t_n^+),\quad [U]^n:=U^n_+-U^n\,.
\]
For our error analysis, we recast the DG FV
method using the global bilinear form:
\begin{equation}\label{eq: GN def}
G(v, w):=\iprod{v^0_+ ,w^0_+}+\sum_{j=1}^{N-1}\iprod{[v]^j,w^j_+}
    +\sum_{j=1}^N\int_{I_j}\,\iprod{v',w}\,dt.
\end{equation}
Integration by parts yields an alternative expression for the bilinear form as
\begin{equation}\label{eq: GN dual}
G(v,w)=\iprod{v^N,w^N}-\sum_{j=1}^{N-1}\iprod{v^j,[w]^j}
    -\sum_{j=1}^N\int_{I_j}\iprod{v,w'}\,dt\,.
\end{equation}
One can see that the local DG FV scheme  \eqref{eq: DG In} holds if and only if $U\in\W$
satisfies
\begin{equation}\label{eq: DG global}
G(U, X)+\int_0^{T} A_h\bigl(\Ba U, \Pi_h^*X\bigr)\,dt=\iprod{U^0 ,X^0_+}+\int_0^{T}\iprod{f, X}\,dt
    \quad\forall ~X\in\W.
\end{equation}
Since the exact solution $u$ of \eqref{eq: fractional diffusion} satisfies \eqref{consistency}
and the identity $[u]^j=0$ for all~$j$, it follows
that
\[
G(u, X)+\int_0^{T} A_h\bigl(\Ba u,\Pi_h^* X\bigr)\,dt=\iprod{u_0,X^0_+}+\int_0^{T}\iprod{f,X}\,dt\,.
\]
Thus, the following  Galerkin orthogonality property follows at once
\begin{equation}\label{eq: GN orthog}
G(U-u,X)+\int_0^{T} A_h\bigl(\Ba (U-u),\Pi_h^* X\bigr)\,dt=\iprod{U^0 -u_0,X^0_+}\quad \forall ~X\in\W.
\end{equation}
\section{Error analysis of the DG FV scheme}\label{sec: error fully discrete numerical method}
\se
To estimate the error from the time-space discretization, we start from the following decomposition:
\begin{equation}\label{eq: U-u=psi+eta+Pi xi}
U-u= (U-\pw \bar R_h u) + (\pw u-u) + \pw(\bar R_h u-u)=:\psi+\eta+\pw\xi,
\end{equation}
where $ \bar R_h $ is the finite volume elliptic projection
operator defined as in (\ref{bar Ritz}), and $\pw : C^0(\overline I_{n};L^2(\Omega)) \to
C^0(\overline I_{n};\p_1(L^2(\Omega))$  is the local (in time) $L^2$-projection defined for $1\le n\le N$
 by
 \begin{equation}\label{eq: Pi properties}
 \pw v(t_n)-v(t_n)=0~~{\rm and}~~\int_{I_n} (\pw v-v,w)\,dt=0~~\forall~w \in L_2(\Omega).
 \end{equation}
The  projection $\pw$  satisfies the error property (Equation 25, Mustapha \& McLean (2011)):
\begin{equation}\label{projection error}
\|\eta\|_{I_n}\le 2k_n\int_{I_n}\|u''(t)\|\,dt\;\;\;\mbox {for}\;\; 1\le n\le N,
\end{equation}
and also the following  error bound property which involves the fractional derivative operator $\Ba $:
\begin{equation}\label{estimate of fractional eta}
\Big|\int_0^T A(\Ba  \eta,\Pi_h^* X) \,dt\Big| \leq
    C\,\|X\|_{J}\,{\bf E},
 \end{equation}
where
\begin{equation} \label{eq: E}
 {\bf E}:= k_1^{\alpha}\int_{I_1} \| Au'\|\,dt+\sum_{n=2}^N k_n^{1+\alpha}\int_{I_n} \| Au''\|\,dt\,.\end{equation}
 Noting that, we used in \eqref{estimate of fractional eta} the following notations
\[
  \|v\|_J:=\max_{n=1}^N \|v\|_{I_n}\quad {\rm with}\quad \|v\|_{I_n}:=\sup_{t\in I_n}\|v(t)\|\,.
\]
 For the proof of \eqref{estimate of fractional eta}, we refer to (Lemma 2, Mustapha \& McLean (2011))
 in addition, to the use of the stability property of $\Pi^*_h$ in Lemma \ref{lem: norm equivlence} (iii).

Next, we estimate  $\psi$. In our proof, we use the following spatial discrete
analogue of \eqref{positive definite property of Ba}:
\begin{equation}\label{eq: discrete positivity}
\int_0^{T} A_h\bigl(\Ba  \chi,\Pi_h^* \chi \bigr)\,dt =
\int_0^{T} A \bigl(\Ba  \chi,\chi \bigr)\,dt \ge c_\alpha T^{\alpha-1}\int_0^{T}\|\nabla \chi(t)\|^2\,dt\quad \forall~~\chi \in \W,\end{equation}
which follows from the identity \eqref{sup-1} and the coercivity property in \eqref{positive definite property of Ba}.
\begin{lemma}\label{lem: ||psi||^2}
Given $\psi= U-\pw \bar R_h u,$ the following estimate holds
\[
\|\psi\|_{J} \le  C\|U^0-\bar R_h u_0\|+\int_0^T\|\xi'(t)\|\,dt+
C\,{\bf E},
\]
where ${\bf E}$ is defined in \eqref{eq: E}.
\end{lemma}
\begin{proof}
The Galerkin orthogonality property~\eqref{eq: GN orthog} along with  the decomposition
\eqref{eq: U-u=psi+eta+Pi xi} and
the identity  \eqref{sup-1} implies that
\begin{multline}\label{eq: orthog step 1}
G(\psi, X)+\int_0^{T} A_h\bigl(\Ba  \psi,\Pi_h^* X\bigr)\,dt
=\iprod{U^0 -u_0,X^0_+}\\
-G(\pw \bar R_h u-u, X)
-\int_0^{T} A \bigl(\Ba  (\pw \bar R_h u-u),X\bigr)\,dt\quad\forall~X\in\W\,.
\end{multline}
Since $\int_{I_n}\iprod{\pw \bar R_h u-\bar R_h u,X'}\,dt=0$ by~\eqref{eq: Pi properties},
integration by parts yields
\begin{align*}
\int_{I_n}\iprod{\pw \bar R_h u-u,X'}\,dt&
    =\int_{I_n}\iprod{\xi,X'}\,dt
    =\iprod{\xi^n,X^n}-\iprod{\xi^{n-1},X^{n-1}_+}
    -\int_{I_n}\iprod{\xi',X}\,dt\,.
\end{align*}
 Hence,  by the alternative formulation  of $G$ in \eqref{eq: GN dual} with $(\pw \bar R_h u-u)(t_n)=\xi^n$
 (by the definition of $\pw$),
\begin{align*}
G(\pw \bar R_h u-u, X)
  &  =
    \iprod{\xi^N,X^N}-\sum_{n=1}^{N-1}\iprod{\xi^n,[X]^n}
    -\sum_{n=1}^N\int_{I_n}\iprod{\pw \bar R_h u-u,X'}\,dt
    \\
 &   =
    \sum_{n=1}^{N}\iprod{\xi^n,X^n}-\sum_{n=2}^{N}\iprod{\xi^{n-1},X^{n-1}_+}
    -\sum_{n=1}^N\int_{I_n}\iprod{\pw \bar R_h u-u,X'}\,dt
    \\
&=
\iprod{\xi^0,X^0_+}
    +\sum_{n=1}^N \int_{I_n}\iprod{\xi',X}\,dt.
\end{align*}
On the other hand, by   \eqref{sup-1}, the following explicit representation of   $\pw \bar R_h u:$
 \[\pw \bar R_h u(t) =\bar R_h u(t_j)+\Big(\bar R_h u(t_j)-k_j^{-1}\int_{t_{j-1}}^{t_j} \bar R_h u(s)\,ds\Big)\frac{2}{k_j}(t-t_j)\quad{\rm for}~~t\in I_j,\]
 the definition of the projection $\bar R_h u$, and  the identity  \eqref{eq: h2 identity}, we notice that
\[
A\bigl(\Ba  \pw \bar R_h u,\Pi_h^* X\bigr)
=A_h\bigl(\Ba  \pw \bar R_h u,\Pi_h^* X\bigr)    =A_h\bigl(\Ba \pw u,\Pi_h^* X\bigr)=(\Ba  \cL\pw u,\Pi_h^* X)\quad {\rm on}~~I_n\,.
\]
Substitute the above contribution  in \eqref{eq: orthog step 1} and use again  the identity \eqref{sup-1} to  obtain
\[
G(\psi,X)+\int_0^{T} A \bigl(\Ba  \psi, X\bigr)\,dt=\iprod{U^0 -\bar R_h u_0,X^0_+}
    -\sum_{j=1}^N \int_{I_j} \bigl(\iprod{\xi',X}+(\Ba  \cL\eta,\Pi_h^* X)\bigr)\,dt.
\]
To proceed in our proof, we choose   $X=\psi$, and use the positivity property of $\Ba$ in \eqref{positive definite property of Ba}) to arrive at
\[
G(\psi,\psi)\le \iprod{U^0 -\bar R_h u_0,\psi^0_+}
    -\sum_{j=1}^N \int_{I_j} \bigl(\iprod{\xi',\psi}+(\Ba  \cL\eta,\Pi_h^* \psi)\bigr)\,dt.
\]
However, by the definition of $G$ given in \eqref{eq: GN def},
\begin{align*}
G(\psi, \psi)&=|\|\psi^0_+\||^2+\sum_{n=1}^{N-1}\iprod{[\psi]^n,\psi^n_+}
    +\frac{1}{2}\sum_{n=1}^N[|\|\psi^n\||^2-|\|\psi^{n-1}_+\||^2]\\
    &=\frac{1}{2}\Big(\sum_{n=1}^{N-1}|\|[\psi]^n\||^2
    +|\|\psi^N\||^2+|\|\psi^0_+\||^2\Big)
    \ge \frac{1}{8}|\|\psi\||_J^2,
\end{align*}
where in the last step we used the fact that for $1\le n\le N$, $\|\psi\|_{I_n}=\max\{\|\psi^n\|,\|\psi^{n-1}_+\|\}$ and the following inequality $\|\psi^{n-1}_+\|^2=\|\psi^{n-1}+[\psi]^{n-1}\|^2\le 2\|\psi^{n-1}\|+2\|[\psi]^{n-1}\|^2\,.$ Therefore,
\[
|\|\psi|\|^2_{J} \le C|\iprod{U^0 -\bar R_h u_0,\psi^0_+}|+C\Big|\int_0^T \bigl[(\xi',\Pi_h^* \psi)+(\Ba  \cL\eta,\Pi_h^* \psi)\bigr]\,dt\Big|\,.
\]
By the  stability property of $\Pi_h^*$ and the equivalence of the norms $\|\cdot\|$ and $|\|\cdot\||$ on  $ V_h$, Lemma \ref{lem: norm equivlence},
\[
\|\psi\|^2_{J} \le  C
    \|\psi\|_{J}\Big(\|U^0 -\bar R_h u_0\|+\int_0^T\|\xi'\|\,dt\Big)+C
    \Big|\int_0^T(\Ba  \cL\eta,\Pi_h^* \psi)\,dt\Big|,
   \]
and finally, the desired result follows after substituting the estimate \eqref{estimate of fractional eta} in the above inequality.
\end{proof}

In the next theorem, we prove  the main convergence results of the fully-discrete scheme. Following Mustapha \& McLean (2011, 2013), we employ time graded meshes based on concentrating the time-steps near $t=0$ to compensate the singular behavior of $u$ of problem \eqref{eq: fractional diffusion}.

To this end,  for a chosen mesh grading
parameter~$\gamma\ge1$, we assume that
\begin{equation}\label{eq: tn standard}
t_n=(n/N)^\gamma T\quad\text{for $0\le n \le N$.}
\end{equation}
It is clear that
for $2\le n\le N$, this time mesh has the following properties:
\begin{equation}\label{eq:kn mesh}\gamma 2^{1-\gamma}k\, {t_{n}}^{1-\frac {1}{\gamma}}\leq {k_{n}}\leq
\gamma\,k\, {t_{n}}^{1-\frac {1}{\gamma}} ~~{\rm and}~~t_{n}\leq 2^\gamma {t_{n-1}}.
\end{equation}
Under suitable regularity assumptions on the solution $u$, we achieve a convergence rate of order $h^2+k^{1+\alpha}$ in the $L^\infty(L^2)$-norm, that is, optimal in space but suboptimal in time. However, the numerical results demonstrate optimal rates of convergence in both variables, in the stronger $L^\infty(L^\infty)$-norm.
\begin{theorem}\label{thm: fully discrete}
Assume that  the solution $u$ of ~\eqref{eq: fractional diffusion}  has the regularity properties \eqref{eq: regularity 2}
 in addition to following property
\begin{equation}\label{eq: regularity u}
\|u'(t)\|+t\|u''(t)\|+t^{\alpha}\|Au'(t)\|+t^{1+\alpha}\|Au''(t)\|\le Mt^{\sigma-1}\quad\text{for $0<t\le T$}
\end{equation}
for some $M,\,\sigma>0$. Assume that the initial data  $u_0 \in H^1(\Omega)$.  Let $U$ be the DG FV solution defined by \eqref{eq: DG In} with $U^0=\bar R_h u_0$.
Then, there is a positive constant $C$,
  depends on  $\alpha$, $T$, $M$, $\sigma$ and $\gamma$, such that
\[
\|U-u\|_{J}\le Ch^2\left(1+ |{\rm log}\; h|\right)+Ck^{1+\alpha}\quad {\rm for}~~\gamma>(1+\alpha)/\sigma\,.\]
\end{theorem}
\begin{proof}
Using the
decomposition~\eqref{eq: U-u=psi+eta+Pi xi},  the stability property of the time projection $\pw$:
$\|\pw \xi\|_{J}\le 3\|\xi\|_{J}$ (see Equation 26, Mustapha \& McLean (2011)),  the inequality  $\|\xi\|_J \le \|\xi_0\|+\int_0^T \|\xi'(t)\|\,dt$,
the achieved estimate
of $\psi$ in Lemma \ref{lem: ||psi||^2}, and the bound in \eqref{eq: bound of xi}, we observe that
\begin{align*}
\|U-u\|_{J} &\le  \|\psi\|_{J}+\|\eta\|_{J}+3\|\xi\|_{J}\\
&\le
 \|\eta\|_{J}+C\Big(\|\xi_0\|+\int_0^T\|\xi'(t)\|\,dt+
{\bf E})\\
&\le \|\eta\|_{J}+C\, h^2\left(1+ |{\rm log}\; h|\right)+
C{\bf E},
\end{align*}
where ${\bf E}$  is defined in \eqref{eq: E}. From  the bound of $\eta$ in \eqref{projection error}, the regularity assumption~\eqref{eq: regularity u}, the time mesh property  \eqref{eq:kn mesh}, and the corresponding  grading mesh assumption $\gamma>(1+\alpha)/\sigma$,
\begin{multline*}
\|\eta\|_J\le \int_{I_1}\|u'(t)\|\,dt+C\max_{n=1}^N k_n\int_{I_n}\|u''(t)\|\,dt\le C\, t_1^\sigma+ C\max_{n=1}^N k_n^2t_n^{\sigma-2}\\
\le C\, k^{\gamma\sigma}+ C\max_{n=1}^N k_n^{1+\alpha}t_n^{\sigma-(1+\alpha)}
    \le C\, k^{\gamma\sigma}+ Ck^{1+\alpha}\max_{n=1}^N  t_n^{\sigma-(1+\alpha)/\gamma}\le
 C\,   k^{1+\alpha}.
\end{multline*}
In a similar fashion, we estimate ${\bf E}$ as
\begin{align*}
{\bf E}&\le C\int_{I_1}t^{\sigma-1}\,dt
    +C\sum_{j=2}^N  k^{1+\alpha}\int_{I_j}t^{(1-1/\gamma)(1+\alpha)}
        \|Au''(t)\|\,dt\\
    &\le Ct_1^\sigma+ Ck^{1+\alpha}\sum_{j=2}^N \int_{I_j}t^{\sigma-(1+\alpha)/\gamma-1}\,dt\\
    &\le Ck^{\gamma\sigma}+Ck^{1+\alpha}\int_{t_1}^{T}t^{\sigma-(1+\alpha)/\gamma-1}\,dt
    \le C    k^{1+\alpha}\quad {\rm for}~~\gamma>(1+\alpha)/\sigma\,.
\end{align*}
Therefore, to complete the proof, we combine the above estimates.
\end{proof}

\section{Numerical Experiments}\label{sec: numerical experiment}
\se
 \begin{figure}
 \begin{center}
 \includegraphics*[width=7.0cm,height=5.0cm]{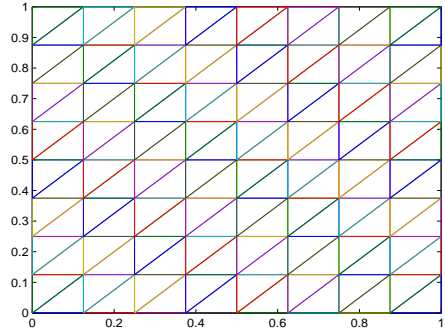}
    \caption{The FV element mesh   $\cT_h$ with $M=8$.}
      \label{Fig: spatial mesh}
   \end{center}
\end{figure}
In this section, we present some numerical tests  to
 validate our theoretical predictions from the fully-discrete DG FV scheme \eqref{eq: DG global}. We actually demonstrate that our achieved  convergence rates are pessimistic and also the imposed regularity assumptions in our error analysis are not sharp. More precisely, for a sufficiently graded time meshes of the form \eqref{eq: tn standard},  our  tests reveal optimal convergence rates in both time and space (that is, of order $O(k^2+h^2)$  in the stronger $L^\infty(L^\infty)$-norm under weaker regularity conditions. However, Theorem \ref{thm: fully discrete} suggest $O(k^{\alpha+1}+h^2)$ rates of convergence  in the $L^\infty(L^2)$-norm.

In our test example,  we consider  $\Omega=(0,1)\times (0,1)$ and   $[0, T ] = [0, 1]$ in the  fractional diffusion problem \eqref{eq: fractional diffusion}.   We choose the initial data $u_0$ and the source term $f$ such that the exact solution $u(x,t)=t^{\alpha}\sin(\pi x)\sin(\pi y),$
and therefore our regularity assumptions \eqref{eq: regularity 2} and \eqref{eq: regularity u} hold for~$\sigma=\alpha$.

 We employ  time meshes of the form \eqref{eq: tn standard} for various choices of the mesh grading parameter $\gamma \ge 1$.
Let $\cT_h$ be a family of uniform (right-angle)  triangular mesh  of the domain  $\Omega$ with diameter ~$h=\sqrt{2}/M$, see Figure \ref{Fig: spatial mesh}. For measuring the error in our numerical solution, we introduce a time finer grid
 \[
\Grid_{N,m}=\{\,t_{j-1}+qk_j:1\le j\le N, 0\le q\le m\,\}
\]
  and let ${\mathcal N}_h$ be the set of all triangular nodes of the mesh family  $\cT_{h_s}$ where the diameter $h_s$ is half the diameter of the finest mesh $\cT_h$ in  our spatial iterations, for instance, $h_s=\sqrt{2}/320$ in Table \ref{table: spatial error for alpha= 0.4 0.75}. Define the following discrete-time-space maximum norm
\[
|\|v\||_{d,m}:=\max\{|v({\bf x},t)|,~({\bf x},t)\in {\mathcal N}_h \times \Grid_{N,m}\}\,.
\]
Thus, for large values of $N, M$ and $m$,  $|\|U-u\||_{d,m}$
approximates the error measured in~$L^\infty(L^\infty)$.

To demonstrate the convergence rates from the spatial discretization by FV method, we  refine the time steps so that the FV errors  are dominant.
 This is achieved by fixing the ratio $\frac{k^{1+\alpha}}{h^2}$ to a given number $<1$. Hence, ignoring the logarithmic term, by  Theorem \ref{thm: fully discrete},  an error of order $O(h^2)$  in the $L^2$-norm is expected. Noting here, for the semi discrete FV scheme \eqref{semi}, we proved in Theorem \ref{sup-conv-2} an $O(h^2)$ rate of convergence  in the stronger  $L^\infty(L^\infty)$-norm under the assumption that the solution $u$ of \eqref{eq: fractional diffusion} is in $H^1(H^3)\cap L^\infty(W^{2,\infty})$. Indeed, this assumption holds  for $\alpha>1/2$ in the current example. However, the numerical results in Table \ref{table: spatial error for alpha= 0.4 0.75} illustrate optimal $O(h^2)$ rates of convergence  for both $\alpha<1/2$ and $\alpha>1/2.$   So, the imposed regularity assumptions   may not be practically required.

To illustrate the convergence rates from time discretization by piecewise linear DG method, we refine the spatial  FV meshes so that the time-stepping error dominates the spatial error.   For $\alpha =0.6$, we observe from Table \ref{Time error for alpha=0.6} convergence rates of order $O(k^{\min\{2,\alpha\gamma\}})$ which is optimal for $\gamma\ge 2/\alpha$. However, a suboptimal convergence of order $O(k^{1+\alpha})$ is proved in  Theorem \ref{thm: fully discrete} assuming that $\gamma>(1+\alpha)/\alpha$.
\begin{table}
\renewcommand{\arraystretch}{0.9}
\begin{center}
\begin{tabular}{|r|rr|rr|}
\hline $M$&\multicolumn{2}{c|}{$\alpha=0.4$}
&\multicolumn{2}{c|}{$\alpha=0.75$} \\
\hline
 10  &3.6522e-02&        &3.1396e-02&       \\
 20  &9.6096e-03&  1.9262&8.2601e-03& 1.9263\\
 40  &2.4235e-03&  1.9873&2.0831e-03& 1.9874\\
 80  &5.8239e-04&  2.0570&5.0059e-04& 2.0570\\
120  &2.4903e-04&  2.0953&2.3145e-04& 1.9026\\
160  &1.3923e-04&  2.0211&1.2713e-04& 2.0827\\
 \hline
\end{tabular}
\caption{The error $|\|U-u\||_{d,10}$ and the spatial rates of convergence for  different choices of $\alpha$.}
\label{table: spatial error for alpha= 0.4 0.75}
\end{center}
\end{table}

\begin{figure}
 \begin{center}
 \includegraphics*[width=9.0cm,height=5.0cm]{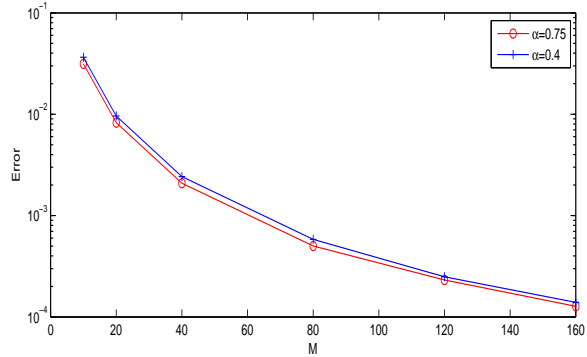}
    \caption{The FV error   for   $\alpha=0.4$ and $\alpha=0.75$.}
      \label{fig: FV error}
   \end{center}
\end{figure}

\begin{table}
\begin{center}
\renewcommand{\arraystretch}{0.9}
\begin{tabular}{|r|rr|rr|rr|rr|}
\hline $N$&\multicolumn{2}{c|}{$\gamma=1$}
&\multicolumn{2}{c|}{$\gamma=2$}
&\multicolumn{2}{c|}{$\gamma=3.4$} \\
\hline
 10& 1.0313e-02&       & 3.2357e-03&        & 2.0414e-03&     \\
 20& 7.2124e-03& 0.5159& 1.5719e-03&  1.0416& 6.2591e-04& 1.7055\\
 40& 5.0788e-03& 0.5060& 7.3189e-04&  1.1028& 1.7878e-04& 1.8078\\
 60& 4.1411e-03& 0.5034& 4.6047e-04&  1.1428& 8.4038e-05& 1.8619\\
 80& 3.5604e-03& 0.5252& 3.2953e-04&  1.1630&           & \\
 \hline
\end{tabular}
\caption{The error $|\|U-u\||_{d,10}$ and the temporal convergence rates for $\alpha=0.6$ with different choices of the mesh grading parameter $\gamma$.  }
\label{Time error for alpha=0.6}
\end{center}
\end{table}
	
\end{document}